\documentclass{amsart}
\usepackage{amsmath, amsthm, amscd, amsfonts, amssymb}

\makeatletter \oddsidemargin.9375in \evensidemargin \oddsidemargin
\newtheorem{thm}{Theorem}[section]
\newtheorem{lem}[thm]{Lemma}
\newtheorem{pro}[thm]{Proposition}
\newtheorem{cor}[thm]{Corollary}
\theoremstyle{definition}

\theoremstyle{remark}
\newtheorem{rem}[thm]{Remark}
\numberwithin{equation}{section}

\begin{document}
	\title[Module and Hochschild  cohomology]{Module and Hochschild  cohomology of certain semigroup algebras}
	\author[A. Shirinkalam, A. Pourabbas and M. Amini]{ A. Shirinkalam,  A. Pourabbas and M. Amini}

	\address{Faculty of Mathematics and Computer
		Science, Amirkabir University of Technology, 424 Hafez Avenue,
		Tehran 15914, Iran}
	\email{shirinkalam\_a@aut.ac.ir}
		\email{arpabbas@aut.ac.ir}
		\address{Department of Mathematics, Tarbiat Modares University, P.O Box 14115-175, Tehran, Iran}
		\address{School of Mathematics, Institute for Research in Fundamental Sciences (IPM), P.O.Box: 19395-5746, Tehran, Iran}
	\email{mamini@modares.ac.ir, mamini@ipm.ir}

	

\begin{abstract}
We study the relation between module and Hochschild cohomology groups of Banach algebras with a compatible module structure. More precisely, we show that for every commutative Banach
 $ \mathcal{A} $-$  \mathfrak{A}$-bimodule $ X $ and  every $ k  \in \mathbb{N}$, the seminormed spaces
 $ \mathcal{H}^{k}_{ \mathfrak{A}} ( \mathcal{A},X^*)$ and  $ \mathcal{H}^k ( \frac{\mathcal{A}}{J}, X^*) $ are isomorphic, where $ J $ is the closed ideal of $ \mathcal{A} $ generated by the elements of the form $ a (\alpha \cdot b)-(a\cdot \alpha)b$ with $ a,b \in \mathcal{A} $ and $\alpha \in \mathfrak{A}. $

As an example, we calculate the module cohomologies of inverse semigroup algebras with coefficients in some related function algebras. In particular, we show that for an inverse semigroup   $ S $   with  the set of idempotents $ E $, when $\ell^1(E)  $ acts on $\ell^1(S)  $ by multiplication from right and trivially from left, the first module cohomology
$\mathcal{H}^1_{\ell^1(E)} (\ell^1(S), \ell^1(G_S)^{(2n+1)})$ is trivial for each $ n \in \mathbb{N} $. As a consequence we conclude that the second module cohomology  $\mathcal{H}^2_{\ell^1(E)} (\ell^1(S),\ell^1(G_S)^{(2n+1)})$ is a Banach space, where $ G_S $ is the maximal group homomorphic image of $ S $.

\end{abstract}
\maketitle
\section{Introduction}
\label{intro}
Johnson in his seminal work  \cite{J} showed that a locally compact group $ G $ is amenable  if and only if the group algebra $ L^1(G) $ is amenable as a Banach algebra.
The analogous results on amenability of semigroup algebras were first obtained by Duncan and Namioka in \cite{dunk}. They  showed that for a discrete  inverse semigroup $S$ with set of idempotents $E$, amenability of $ \ell^1(S) $ implies that $ E $ is finite and all maximal subgroups of $S$ are amenable (see \cite{DLS} for more general results in this direction).

In an attempt to recover a version of Johnson's result for (inverse) semigroups, the concept of module amenability introduced by the third author in \cite{A}.
He showed that for an inverse semigroup $ S $ with set of idempotents $E$, if $ \ell^1(S) $ is considered as a module over $ \ell^1(E) $ (with trivial left action and multiplication as the right action) then module  amenability of the semigroup algebra $ \ell^1(S) $ (see below for a definition) is equivalent to amenability of $S$. On the other hand, Bowling and Duncan devote part of \cite{BD2} to calculate the cohomology groups of   $ \ell^1(S) $ with coefficients in certain function spaces on $S$, including the algebra  $ \ell^\infty(S) $ of bounded functions on $S$. The most interesting case is when the cohomology is trivial.

The second cohomology group of a Banach algebra involves information about splitting of extensions of the underlying Banach algebra \cite{H}. When the higher cohomologies are not trivial, it is usually desirable to show that they are Banach spaces.
In \cite {I} Ivanov and in \cite {MM} Matsumoto and Morita showed that $ \mathcal{H}^2(\ell^1(G) , \mathbb{C}) $ is a Banach space. Choi, Gourdeau  and White in \cite{CGW}  showed that for a band semigroup $ S $, that is, a semigroup all of whose elements are idempotent, the simplicial cohomology of $ \ell^1(S) $ in each degree is trivial. For a semilattice $ S $, that is, a commutative semigroup $ S $ in which $ s^2=s $ for every $ s \in S $, Dales and Duncan \cite{DD} showed that $ \mathcal{H}^{2}(\ell^1(S),X) =0$, for every Banach $ \ell^1(S) $-bimodule $ X $.

 Finally, Gourdeau, Pourabbas and   White in \cite{GPW} showed that the second simplicial cohomology group of $ \ell^1(S) $ is a Banach space, for   a Clifford   Rees semigroup $S$, and in particular this is the case for $ \mathcal{H}^{2}(\ell^1(N_a),\ell ^\infty (N_a) )  $ where  $ N_a=\lbrace n\in \mathbb{Z}_+ : n \geq a \rbrace $ with $ a  >0 $. As an example on even higher cohomologies, it is shown for a semilattice $ S $ in \cite{GPW} that  $ \mathcal{H}^{3}(\ell^1(S),X)  $ is a Banach space, for every Banach $ \ell^1(S) $-bimodule $ X $.

Motivated by the results of the third author in \cite{A}, the first and the second module cohomology group of inverse semigroup algebras were studied by Nasrabadi and the second author in  \cite{NP2} and \cite{NP}. They showed that for a commutative inverse semigroup $ S $ with the set of idempotents $ E $, when  $ \ell^1(E) $   acts on    $ \ell^1(S) $ by multiplication from both sides, the module cohomology  $ \mathcal{H}^1_{ \ell^1(E)} ( \ell^1(S),\ell^1(S)^{(2n+1)}) $ is trivial, for each $ n \in \mathbb{N} $. Also they showed that for a  commutative or Clifford semigroup $ S $, $ \mathcal{H}^2_{ \ell^1(E)} ( \ell^1(S),\ell^1(S)^{(2n+1)}) $ is a Banach space.

The present paper is motivated by the last two cited papers and is the first attempt to systematically investigate the relation between classical and module cohomologies. We show that, under certain conditions, satisfied by a class of examples, these cohomologies, with appropriate choice of coefficients,  are the same. This in particular allows us to drop the commutativity assumption in the above mentioned results of \cite{NP2} and \cite{NP}, when the action is assumed to be trivial from left.
The paper is organized as follows. In section 1, we give basic definitions and recall certain preliminary facts.
In section 2, we investigate the relation between the classical (Hochschild) and module cohomologies for a Banach algebra with a compatible module action of another Banach algebra.
In section 3, we use the main result of section 2 to show that for an inverse semigroup $ S $ with set of idempotents $E$ and  maximal group homomorphic image $ G_S $, the first module cohomology
$ \mathcal{H}^1_{ \ell^1(E)} ( \ell^1(S),( \ell^1(G_S) )^{(2n+1)})  $
is trivial, and the second module cohomology  $ \mathcal{H}^2_{ \ell^1(E)} ( \ell^1(S),( \ell^1(G_S) )^{(2n+1)}) $ is a Banach space, for each $ n \in \mathbb{N} $.

    We denote by $\otimes $ and $ \hat{\otimes} $, the algebraic and projective tensor product of Banach spaces (algebras), respectively. Also $ \mathcal{L}^n(X,Y) $ (resp. $ \mathcal{L}^n_{ \mathfrak{A}}(X,Y) $) denotes the space of all bounded $  n$-linear (resp. module) maps from $ X $ into $ Y $ (when $ X $ and $ Y $ are $  \mathfrak{A}$-bimodules).
 Let $ \mathcal{A} $ and $  \mathfrak{A}$ be Banach algebras   and let  $ \mathcal{A} $  be a Banach $  \mathfrak{A}$-bimodule. We say that the action of $  \mathfrak{A}$ on  $ \mathcal{A} $ is compatible, if for each $ \alpha \in \mathfrak{A} $ and $ a, b \in  \mathcal{A}$, we have
  $$ \alpha \cdot (ab)=(\alpha \cdot a)b, \quad (ab)\cdot \alpha =a(b \cdot \alpha). $$

 If $  \mathcal{A}$  is a  Banach $  \mathfrak{A}$-bimodule with  the compatible action, then so is the dual space $  \mathcal{A}^* $. Throughout the rest of the paper, we assume that $ \mathcal{A} $  is a Banach $  \mathfrak{A}$-bimodule with compatible actions.

Let $ X $ be a Banach $  \mathcal{A}$-bimodule and a Banach $  \mathfrak{A}$-bimodule such that for every $\alpha \in   \mathfrak{A}, a \in \mathcal{A}  $ and $ x \in X, $
\begin{eqnarray}\label{compatible} 
\alpha \cdot (a\cdot x)=(\alpha\cdot a)\cdot x, \quad a\cdot (\alpha \cdot x)=(a\cdot \alpha)\cdot x, \quad (\alpha \cdot x)\cdot a = \alpha\cdot (x\cdot a), 
\end{eqnarray}
 and the same for the right and two-sided actions. Then we say that $ X $ is a Banach
 $ \mathcal{A} $-$  \mathfrak{A}$-bimodule. If moreover, for each $ \alpha\in   \mathfrak{A}$ and $ x \in X $, $ \alpha\cdot x =x\cdot \alpha $, then $ X $ is called a commutative Banach
 $ \mathcal{A} $-$  \mathfrak{A}$-bimodule. In this case $ X^* $, the dual of $ X $ is also a commutative Banach
 $ \mathcal{A} $-$  \mathfrak{A}$-bimodule with the canonical action.
Let $ Y  $ be another Banach $ \mathcal{A} $-$  \mathfrak{A}$-bimodule, then an $ \mathcal{A} $-$  \mathfrak{A}$-bimodule morphism from $ X $ to $ Y $ is a bounded linear map $ \varphi :X \rightarrow Y $ which is module morphism with respect to both actions, that is,
$$\varphi (a \cdot x)=a\cdot \varphi (x), \quad \varphi (x\cdot a)=\varphi(x)\cdot a, \quad  \varphi (\alpha \cdot x)=\alpha\cdot \varphi (x), \quad \varphi (x\cdot \alpha)=\varphi(x)\cdot \alpha,  $$ for each $\alpha \in   \mathfrak{A}, a \in \mathcal{A}$ and $ x \in X.  $

\begin{rem}
When $ \mathcal{A} $ acts on itself by algebra multiplication, it is not in general a Banach $ \mathcal{A} $-$  \mathfrak{A}$-bimodule,
 since we have not assumed the condition
 \begin{eqnarray}\label{com} 
a  (\alpha \cdot b)= (a\cdot \alpha)b \quad(\alpha \in \mathfrak{A}, a,b \in \mathcal{A}).
\end{eqnarray}

 However, if  $ \mathcal{A} $ is a commutative  $  \mathfrak{A}$-bimodule, then the equation (\ref{com}) holds, so  $ \mathcal{A} $ is a Banach  $\mathcal{A} $-$  \mathfrak{A}$-bimodule.
 \end{rem}

    For a Banach algebra $ \mathcal{A} $, the projective tensor product $ \mathcal{A} \hat{\otimes}\mathcal{A}  $ is a Banach algebra with respect to the multiplication defined on elementary tensors by $$ (a\otimes b)(c\otimes d)=ac \otimes bd, $$ and extended by bi-linearity and continuity \cite {BD}. If  $ \mathcal{A} $ is a Banach  $  \mathfrak{A}$-bimodule, then
      $ \mathcal{A} \hat{\otimes}\mathcal{A}  $ becomes a Banach  $  \mathfrak{A}$-bimodule with the  canonical actions. Let $ I $ be the closed ideal of
       $ \mathcal{A} \hat{\otimes}\mathcal{A}  $ generated by the elements of the form
      $ a\cdot\alpha \otimes b - a\otimes \alpha\cdot b, $ for  $ \alpha \in \mathfrak{A}, a,b \in \mathcal{A} $.
The module projective tensor product $ \mathcal{A} \hat{\otimes}_{\mathfrak{A}}\mathcal{A} $ is the quotient $ \dfrac{\mathcal{A} \hat{\otimes}\mathcal{A}}{I} $ \cite {H}.
Let $ J $ be the closed ideal of $ \mathcal{A} $ generated by $ \pi (I),$ where  $ \pi : \mathcal{A} \hat{\otimes}\mathcal{A}  \rightarrow \mathcal{A} $ is the multiplication map.  This is the same as the closed linear span of  the set of elements of the form $ (a\cdot\alpha)  b - a( \alpha\cdot b ) \:(\alpha \in \mathfrak{A}, a,b \in \mathcal{A} ) $.  Moreover,
$ \mathcal{A} \hat{\otimes}_{\mathfrak{A}}\mathcal{A} $ and $ \frac{\mathcal{A}}{J} $
are $  \mathfrak{A}$-bimodules and the map $ \tilde{\pi}:\mathcal{A} \hat{\otimes}_{\mathfrak{A}}\mathcal{A}\rightarrow \frac{\mathcal{A}}{J}  $ defined by
$\tilde{\pi}(a\otimes b +I):=ab+J  $ extends to an $  \mathfrak{A}$-bimodule morphism \cite {Rif}.

Let $ X $ be a Banach
 $ \mathcal{A} $-$  \mathfrak{A}$-bimodule. A bounded linear map $ D:\mathcal{A} \rightarrow X $
 is called an $  \mathfrak{A}$-module derivation if it is a derivation and  an $  \mathfrak{A}$-module map. When $ X $ is a commutative $ \mathcal{A} $-$  \mathfrak{A}$-bimodule, for $ x\in X $, the  map $ \hbox{ad}_x:\mathcal{A} \rightarrow X  $
 defined by $ \hbox{ad}_x (a):=a\cdot x- x\cdot a $ is a module derivation, called an inner module derivation.

Let $ X $ be a commutative Banach
 $ \mathcal{A} $-$  \mathfrak{A}$-bimodule. The space of all  $  \mathfrak{A}$-module derivations from   $  \mathcal{A}$ to $ X $ is denoted by $ \mathcal{Z}^1_{ \mathfrak{A}} ( \mathcal{A},X)$ and the subspace of  inner module derivations is denoted by $ \mathcal{B}^1_{ \mathfrak{A}} ( \mathcal{A},X)$. The first $\mathfrak{A}$-module cohomology group of $ \mathcal{A} $ with coefficients in $ X $ is defined as the quotient seminormed space
$$\mathcal{H}^1_{ \mathfrak{A}} ( \mathcal{A},X):=\dfrac{\mathcal{Z}^1_{ \mathfrak{A}} ( \mathcal{A},X)}{\mathcal{B}^1_{ \mathfrak{A}} ( \mathcal{A},X)}.  $$
If $ \mathcal{B}^1_{ \mathfrak{A}} ( \mathcal{A},X) $ is closed in $ \mathcal{Z}^1_{ \mathfrak{A}} ( \mathcal{A},X) $ in the norm topology, then $ \mathcal{H}^1_{ \mathfrak{A}} ( \mathcal{A},X) $ is a Banach space. Similar condition applies to higher cohomologies (compare with \cite{NP}).
\section{Relation between module cohomology and Hochschild cohomology}
\label{sec:1}
Let
 $ \mathcal{A} $ and $  \mathfrak{A}$ and the closed ideal $J$ be as in the previous section, and  $ X $ be a Banach
 $ \mathcal{A} $-$  \mathfrak{A}$-bimodule. In this section, we provide sufficient conditions for $ \mathcal{H}^1_{ \mathfrak{A}} ( \mathcal{A},X)$ to be isomorphic to $ \mathcal{H}^1 ( \frac{\mathcal{A}}{J},X) $ as seminormed spaces.

 We say that  the action of $  \mathfrak{A}$  on $X$ is trivial from left, if
  for  every $ \alpha \in \mathfrak{A} $ and $ x \in X, \alpha \cdot x =f(\alpha)x$,
  where $ f $ is a character on  $ \mathfrak{A} $.

  \begin{rem}\label{r1}
If $ X $ is a Banach
 $ \mathcal{A} $-$  \mathfrak{A}$-bimodule, the equation (\ref{compatible}) shows that $ J\cdot X = X \cdot J = 0 $. In particular, any action of $ \mathcal{A} $ on $ X $ induces an action of $ \frac{\mathcal{A}}{J} $ on $ X $ given by $ (a+J)\circ x=a\cdot x$ and $ x\circ (a+J)=x\cdot a, $ for  $ a \in \mathcal{A} $ and $ x\in X $, turning  $ X $ into a Banach
 $ \frac{\mathcal{A}}{J} $-$  \mathfrak{A}$-bimodule. Conversely, if  $ X $ is a Banach
 $ \frac{\mathcal{A}}{J} $-$  \mathfrak{A}$-bimodule, then it is a Banach
 $ \mathcal{A} $-$  \mathfrak{A}$-bimodule with the obvious action.
\end{rem}

For the rest of this section, we assume that $\mathcal{A}  $ is a Banach $\mathfrak{A}$-bimodule and $ X $ is a commutative Banach
 $ \mathcal{A} $-$  \mathfrak{A}$-bimodule, both with a left trivial action of $\mathfrak{A}$, via the same character $f$ on  $ \mathfrak{A} $. We also assume that the Banach algebra $A/J$ is unital. In the next section, we provide examples satisfying all these conditions.

 \begin{pro}\label{main} The Banach spaces $\mathcal{Z}^1_{ \mathfrak{A}} ( \mathcal{A},X)$ and $\mathcal{Z}^1 ( \frac{\mathcal{A}}{J},X)  $ are
isometrically  isomorphic.
\end{pro}
\begin{proof}
Consider the map $ \rho :\mathcal{Z}^1_{ \mathfrak{A}} ( \mathcal{A},X)\rightarrow \mathcal{Z}^1 ( \frac{\mathcal{A}}{J},X)  $ defined   by $ \rho(D)(a+J)=D(a), $ for  $ D \in  \mathcal{Z}^1_{ \mathfrak{A}} ( \mathcal{A},X)$. This is well defined, since each module derivation in $ \mathcal{Z}^1_{ \mathfrak{A}} ( \mathcal{A},X) $ vanishes on $ J $.
Moreover, for each  $ a \in \mathcal{A} $, $  \Vert a+J \Vert = \inf _{j \in J} \Vert a+j \Vert \leq \Vert a \Vert $, hence
$$ \Vert \rho(D) \Vert = \sup_{\Vert a+J \Vert \neq 0} \dfrac{\Vert D(a) \Vert}{\Vert a+J \Vert} \geq \sup_{\Vert a \Vert \neq 0} \dfrac{\Vert D(a) \Vert}{\Vert a  \Vert}= \Vert D \Vert. $$
Conversely, if $ \Vert a \Vert \leq 1 $, then for each $ \delta >0 $, there exists $ j \in J $ such that
$$ \Vert a+j \Vert < \Vert a+J \Vert+\delta \leq 1+ \delta. $$
 Hence
 $$ \Vert D \Vert \geq \dfrac{\Vert D(a+j) \Vert}{\Vert  a+j \Vert} \geq \dfrac{\Vert \rho(D)(a+J)  \Vert}{\Vert a+J   \Vert + \delta}. $$ Therefore, we have
 $$ \Vert \rho(D)(a+J) \Vert \leq \Vert D\Vert (\Vert a+J    \Vert  +\delta) \leq \Vert D \Vert (1+\delta).$$ Taking supremum,
 $  \Vert \rho(D) \Vert \leq \Vert D \Vert (1+\delta),$ thus $ \rho $ is an isometry. 
 
 Next we show that $ \rho $ is surjective. Suppose  $ D^\prime \in  \mathcal{Z}^1 ( \frac{\mathcal{A}}{J},X) $. Define $ D:\mathcal{A} \rightarrow X $  by $ D(a):=D^\prime (a+J) $, for  $ a \in \mathcal{A} $. Then $ D $ is bounded linear map and
 \begin{align*}
D(ab)
&= D^\prime (ab+J) \\
&=D^\prime ((a+J)(b+J))  \\
&= D^\prime(a+J)\circ(b+J) +(a+J)\circ D^\prime (b+J) \\
&=D(a)\cdot b+ a\cdot D(b)\quad \qquad \qquad \qquad\text{(by Remark \ref {r1})},
\end{align*}
Since the left actions of $ \mathfrak{A}$  on both $ \mathcal{A} $ and $ X $  are  trivial,
$$D(\alpha \cdot a) = D  (f(\alpha)a) =f(\alpha) D(a) =\alpha \cdot D(a), $$
for  $ \alpha \in \mathfrak{A} $.
Also by
\cite [Lemma 3.1]{ABE},
  $ f(\alpha)a-a\cdot\alpha \in J $, thus $ f(\alpha)a +J = a\cdot\alpha+ J $, for  $ \alpha \in \mathfrak{A}, a \in  \mathcal{A}$. Therefore
\begin{align*}
D(a\cdot \alpha)
&= D^\prime  (a \cdot \alpha +J) \\
&=D^\prime (f(\alpha)a+J)  \\
&=\alpha \cdot D(a)  \\
&= D(a) \cdot \alpha,
\end{align*}
and $ D $ is an  $  \mathfrak{A}$-module derivation.
Now by the definition,  $ \rho (D)=D^\prime, $
hence $ \rho $ is surjective.
\end{proof}

Let $ R $ be a ring and let $ X $ and $ Y $ be $ R $-modules and $ f:X \rightarrow Y $ be an $ R $-module homomorphism. Suppose that  $ X^\prime $ and $ Y^\prime $ are submodules of $ X $ and $ Y $, respectively. If $ f( X^\prime) \subseteq Y^\prime $, then
 $ f $ induces an $ R $-module homomorphism $ \tilde{f}: \frac{X}{ X^\prime}\rightarrow\frac{Y}{Y^\prime} $ given by $\tilde{f} (a+X^\prime) = f(a)+ Y^\prime.   $ If $ f $ is an epimorphism such that $ f(X^\prime)=Y^\prime  $ and ker$ f \subseteq  X^\prime$, then $ \tilde{f} $ is an $ R $-module isomorphism \cite [Corolarry IV.1.8]{sh}. A topological version of this observation holds for Banach modules.

\begin{thm}\label{main1}
The seminormed  spaces
$\mathcal{H}^1_{ \mathfrak{A}} ( \mathcal{A},X)$ and $\mathcal{H}^1 ( \frac{\mathcal{A}}{J},X)  $ are isomorphic.
\end{thm}
\begin{proof}
The isomorphism $ \rho $ in the above proposition maps the inner module derivations on $ \mathcal{A} $ into the inner derivations on $\frac{\mathcal{A}}{J}  $, and so induces a  surjective map
$$ \Phi: \dfrac{\mathcal{Z}^1_{ \mathfrak{A}} ( \mathcal{A},X)}{\mathcal{B}^1_{ \mathfrak{A}} ( \mathcal{A},X)} \rightarrow \dfrac{\mathcal{Z}^1 ( \frac{\mathcal{A}}{J},X)}{\mathcal{B}^1 ( \frac{\mathcal{A}}{J},X)}, $$
given by $$ \Phi (D + \mathcal{B}^1_{ \mathfrak{A}} ( \mathcal{A},X)):= \rho (D)+ \mathcal{B}^1 ( \frac{\mathcal{A}}{J},X).$$ Let us show that  $ \Phi $ is injective. Suppose $ D \in \mathcal{Z}^1_{ \mathfrak{A}} ( \mathcal{A},X) $ is such that $ \rho (D)\in \mathcal{B}^1 ( \frac{\mathcal{A}}{J},X) $. Then
  there exists an element $ x \in X $ such that $\rho (D)(a+J)=(a+J) \circ x- x \circ (a+J) $, for each $a \in \mathcal{A},$
 thus, by Remark \ref {r1},  $$  D(a)= \rho (D)(a+J)=(a+J) \circ x- x \circ (a+J)=a\cdot x -x \cdot a,$$
  that is, $ D \in \mathcal{B}^1_{ \mathfrak{A}} ( \mathcal{A},X) $.  Next we show that  $ \Phi $ is an isometry. Consider  the seminorms $ \varrho $ and $ \sigma $, respectively  on $\mathcal{H}^1_{ \mathfrak{A}} ( \mathcal{A},X),  $ and  $\mathcal{H}^1 ( \frac{\mathcal{A}}{J},X)  $  given by
   $$\varrho(D+\mathcal{B}^1_{ \mathfrak{A}} ( \mathcal{A},X)) := \inf _{x \in X}  \Vert D+ \hbox{ad}_x      \Vert, $$ and $$\sigma( \rho(D)+\mathcal{B}^1 ( \frac{\mathcal{A}}{J},X)) := \inf _{x \in X}  \Vert   \rho(D)+ \tilde{\hbox{ad}}_x     \Vert, $$  where ${\hbox{ad}}_x$ and  $\tilde{\hbox{ad}}_x$ are the corresponding inner (module) derivations on $\mathcal A$ and $\mathcal A/J$, respectively. Then,
    \begin{align*}
\varrho(D+\mathcal{B}^1_{ \mathfrak{A}} ( \mathcal{A},X))
&=\inf _{x \in X}  \Vert D+ \hbox{ad}_x      \Vert \\
&=\inf _{x \in X}  \Vert \rho (D+ \hbox{ad}_x  )    \Vert \qquad (\rho \text{ is an isometry})  \\
&=\inf _{x \in X}  \Vert \rho (D)+ \tilde{\hbox{ad}}_x    \Vert   \\
&=\sigma( \rho (D)+\mathcal{B}^1 ( \frac{\mathcal{A}}{J},X))\\
&=\sigma(\Phi (D+\mathcal{B}^1_{ \mathfrak{A}} ( \mathcal{A},X))).
\end{align*}
Therefore $\Phi$ is an isomorphism of seminorms.
\end{proof}

As a corollary, we get  \cite [Proposition 3.2]{ABE}, which asserts that if $ \mathcal{A} $ is module amenable as an $  \mathfrak{A}$-bimodule with  trivial left action and $  \frac{\mathcal{A}}{J} $ is unital, then $  \frac{\mathcal{A}}{J} $ is amenable.

Let $ X, Y $ be Banach spaces. An operator $ \theta : X \rightarrow Y $ is called a quotient map if $ \theta (\hbox{ball} \,X) $ is dense in $ \hbox{ball} \,Y $. In this case we have the isometric isomorphism  $ \dfrac{X}{\ker \theta} \simeq Y$ (see \cite {LT}).

Let $ \lambda $ be a positive real number. A Banach space $X$ is said to be an
$ \mathfrak{L}_{1,\lambda}$-space if every finite dimensional subspace $M$ of $X$ is contained in a
finite dimensional subspace $N$ of $X$, of dimension $k$, such that whose Banach-Mazur distance
from the space $\ell^{k}_{1} $ is at most $  \lambda$.  This means that there exists an isomorphism $ T: N \rightarrow \ell^{k}_{1} $ such that $ \Vert T \Vert \: \Vert T^{-1} \Vert \leq \lambda. $
If $X$ is a $ \mathfrak{L}_{1,\lambda}$-space for some $  \lambda >1$,
then $X$ is said to be an $ \mathfrak{L}_{1 }$-space. Examples of such spaces are $ L^1(\mu) $ for a measure $ \mu $ and $ C(K)^* $, for a compact space $ K $. We use the fact that the projective tensor product with $ \mathfrak{L}_{1 }$-spaces respect subspaces isomorphically, that is, if $ Y $ is a subspace of $ Z $ and $ X $ is an $ \mathfrak{L}_{1 }$-space, then the projective norm on $ Y \otimes X $ is equivalent to the norm induced by the projective norm on $ Z \otimes X $
\cite{Ry}.

In the next lemma, under some conditions, we identify  the Banach spaces $ \mathcal{A}\hat{\otimes}_{\mathfrak{A}} X $ and $ \frac{\mathcal{A}}{J} \hat{\otimes}X $. Note that by the definition of the module projective tensor product we have $\mathcal{A}\hat{\otimes}_{\mathfrak{A}} X \simeq  \frac{\mathcal{A}\hat{\otimes} X}{I_0 }  $,  where $ I_0 $ is the closed linear span of the set of elements of the form $ a\cdot \alpha \otimes x - a \otimes \alpha \cdot x$ with $a \in \mathcal{A}, \alpha \in   \mathfrak{A}, x \in X$.

Recall that $J$ is a closed ideal of $\mathcal A$ and is equal to the closed linear span of  the set of elements of the form $ (a\cdot\alpha)  b - a( \alpha\cdot b )$ with $\alpha \in \mathfrak{A}, a,b \in \mathcal{A}$. Let $J_0$ be the closed linear span of the set of elements of the form $ a\cdot\alpha - f(\alpha)a$ with $a \in  \mathcal{A}, \alpha \in   \mathfrak{A}. $
Since $ \frac{\mathcal{A}}{J} $ is assumed to be unital, $J_0\subseteq J $. In practice, for most interesting examples (see the next section), one has $J=J_0$.

\begin{lem}\label{tensor}
With the above notation, assume that $J=J_0$ and  $ X $ is an $ \mathfrak{L}_{1 }$-space. Then the Banach spaces  $ \frac{\mathcal{A}}{J} \hat{\otimes}X $ and $\frac{\mathcal{A}\hat{\otimes} X}{I_0 } $ are isometrically isomorphic.
\end{lem}
\begin{proof}
Consider the surjective map $ \varphi :=q \otimes id :\mathcal{A}  \hat{\otimes}X \rightarrow \frac{\mathcal{A}}{J}\hat{\otimes}X$  defined by \\
$\varphi(a\otimes x):=(a+ J )\otimes x $, for $ a \in \mathcal{A}, x \in X $, and  extended by linearity and continuity, where $ q :\mathcal{A} \rightarrow \frac{\mathcal{A}}{J}  $  is the usual quotient map. Since both  $ q $ and $ id $ are quotient maps, so is $ \varphi $  \cite[Proposition 2.5]{Ry} and we have an isometric isomorphism
$$\dfrac{\mathcal{A}  \hat{\otimes}X}{\ker \varphi}\simeq \frac{\mathcal{A}}{J}\hat{\otimes} X.  $$

Let us show that $\ker \varphi =I_0  $. First we show that  $\ker \varphi =J\hat{\otimes} X  $ and then $   J\hat{\otimes} X =I_0$. Consider the short exact sequence   $$\Gamma : \qquad 0 \rightarrow J \hookrightarrow \mathcal{A} \rightarrow \frac{ \mathcal{A}}{J} \rightarrow 0.$$ Since $ X $ is an $ \mathfrak{L}_{1 }$-space,  the closures of $  J\otimes X$ in $ J\hat{\otimes} X $ and $ \mathcal{A}\hat{\otimes} X  $  coincide  \cite [Theorem 2.20]{Ry}.  In particular,
 $ J  \hat{\otimes} X  \subseteq \mathcal{A}\hat{\otimes} X $ is closed. Let us denote this inclusion by $ \iota $. If we show that  the short sequence $$ \Sigma: \qquad J  \hat{\otimes} X  \xrightarrow{\iota} \mathcal{A}\hat{\otimes} X \xrightarrow{\varphi}  \frac{ \mathcal{A}}{J} \hat{\otimes} X \rightarrow 0 $$
 is exact at $ \mathcal{A}\hat{\otimes} X $, we get  $  \ker \varphi = J\hat{\otimes} X. $
Consider the short sequence $$\Sigma^*:\qquad 0 \rightarrow  (\frac{ \mathcal{A}}{J}\hat{\otimes} X)^* \xrightarrow{\varphi^*} ( \mathcal{A} \hat{\otimes} X)^* \xrightarrow{\iota^*} ( J\hat{\otimes} X)^*.$$ This is the same as
$$0 \rightarrow  \mathfrak{L}(\frac{ \mathcal{A}}{J},X^*) \rightarrow  \mathfrak{L}(  \mathcal{A},X^*) \rightarrow \mathfrak{L}(J,X^*),$$  is exact at $ \mathfrak{L}(  \mathcal{A},X^*) $ \cite[Exercise 0.5.4]{H}. Thus
$\Sigma^*$
is exact at  $ ( \mathcal{A} \hat{\otimes} X)^* $. Since
 the image $ J  \hat{\otimes}   X $ of $ \iota $ is closed in $  \mathcal{A} \hat{\otimes}   X $,
 the image of $ \iota^* $ is closed in $ ( J\hat{\otimes} X)^* $. Therefore, $ \Sigma $ is exact at $ \mathcal{A} \hat{\otimes} X $  \cite[Theorem 2.8.31(iii)]{D}, thus $  \ker \varphi = J\hat{\otimes} X. $

 Finally, let us show that $I_0 = J\hat{\otimes} X $. Since the left action of $\mathfrak{A}$ on $X$ is trivial, $ I_0 $ is spanned by the elements of the form $ (a\cdot \alpha - f(\alpha)a )\otimes x $.
Thus $J\otimes X \subseteq I_0$.  Since $ I_0 $ is closed in $\mathcal{A} \hat{\otimes} X$, we have
 $\overline{J\otimes X} \subseteq I_0$, where the closure is taken in $\mathcal{A} \hat{\otimes} X$. Thus $ J \hat{\otimes} X = \overline{J\otimes X} \subseteq I_0 $. The inverse inclusion follows from the assumption that $J_0=J $.
\end{proof}

\begin{rem}\label{re}
It is well-known that the projective tensor product of two $ \mathfrak{L}_{1 }$-space is  an $ \mathfrak{L}_{1 }$-space. Hence if $ \frac{\mathcal{A}}{J} $   is an $ \mathfrak{L}_{1 }$-space, then  by  induction, under the assumptions of  Lemma \ref {tensor}, for every $ k  \in \mathbb{N}$, we conclude that  $$ \underbrace{\mathcal{A}\hat{\otimes}_{\mathfrak{A}} \mathcal{A}\hat{\otimes}_{\mathfrak{A}}\cdots\hat{\otimes}_{\mathfrak{A}}\mathcal{A}}_{k-times}\hat{\otimes}_{\mathfrak{A}}X  \simeq \underbrace{\frac{\mathcal{A}}{J} \hat{\otimes} \frac{\mathcal{A}}{J} \hat{\otimes}\cdots  \hat{\otimes} \frac{\mathcal{A}}{J}}_{k-times} \hat{\otimes}X.  $$
\end{rem}

\begin{rem}\label{act1} 
The space $ Y:=\underbrace{\mathcal{A}\hat{\otimes}_{\mathfrak{A}} \mathcal{A}\hat{\otimes}_{\mathfrak{A}} \cdots\hat{\otimes}_{\mathfrak{A}}\mathcal{A}}_{k-times}\hat{\otimes}_{\mathfrak{A}}X$   is  a Banach    $  \mathfrak{A}$-bimodule with the following actions
$$
(a_1\otimes\cdots\otimes a_k\otimes x)\cdot \alpha := a_1\otimes\cdots\otimes a_k\otimes x\cdot \alpha, \quad
\alpha \cdot  (a_1\otimes\cdots\otimes a_k\otimes x) := \alpha \cdot a_1\otimes\cdots\otimes a_k\otimes x,
$$
and  a Banach $  \mathcal{A}$-bimodule
with the following actions,
 $$(a_1\otimes\cdots\otimes a_k\otimes x)\cdot b := a_1\otimes\cdots\otimes a_k\otimes x\cdot b, $$ and
 \begin{align*}
b \cdot (a_1\otimes\cdots\otimes a_k\otimes x)
&:= ba_1\otimes\cdots\otimes a_k\otimes x \\
&+ \sum^{k-1}_{j=1}b \otimes a_1\otimes\cdots \otimes a_j a_{j+1}\otimes\cdots\otimes a_k \otimes x \\
& +(-1)^{k}   b \otimes a_1 \otimes\cdots \otimes a_k \cdot x, 
\end{align*}
for $  b, a_1,\ldots, a_k \in \mathcal{A}, \alpha \in  \mathfrak{A}, x \in X.$
\end{rem}
Here we extend the result of Theorem \ref{main1} to higher cohomology.
A version of reduction of dimension for module cohomologies is valid (under the standing assumptions of the paragraph before Proposition 2.2) as follows. 
\begin{pro}(Reduction  of dimension)\label{red}
Let $ \frac{\mathcal{A}}{J} $ is  an  $ \mathfrak{L}_{1 }$-space. Under the assumptions of Lemma \ref{tensor}, 
 \begin{enumerate}
\item[(i)]   $  \mathcal{L}_{\mathfrak{A}}^k(\mathcal{A},X^*) \simeq (\underbrace{\mathcal{A}\hat{\otimes}_{\mathfrak{A}} \cdots\hat{\otimes}_{\mathfrak{A}}\mathcal{A}}_{k-times}\hat{\otimes}_{\mathfrak{A}}X     )^* $ is a commutative Banach
 $ \mathcal{A} $-$  \mathfrak{A}$-bimodule, for every  $ k  \in \mathbb{N}$,

\item[(ii)] We have the isomorphism
 $$ \mathcal{H}^{n+k}_{ \mathfrak{A}} ( \mathcal{A},X)\simeq \mathcal{H}^n _{ \mathfrak{A}}(\mathcal{A},\mathcal{L}_{\mathfrak{A}}^k(\mathcal{A},X) ),  $$ of seminormed spaces, for every  $ k,n  \in \mathbb{N}$.
\end{enumerate}

\end{pro}
\begin{proof}
(i) 
We show  that the module actions  defined in Remark \ref{act1} are compatible.
Let us check the case $ k=1 $, the cases $ k>1 $ follow by an easy induction.
Take $ y=b\otimes x \in  \mathcal{A}\hat{\otimes}X, a \in \mathcal{A}$ and $ \alpha \in  \mathfrak{A}$. Let us observe that $ (a \cdot \alpha )\cdot y +I_0 = a \cdot (\alpha \cdot y )+I_0,$ which is evident, since 
\begin{align*}
(a \cdot \alpha )\cdot y - a \cdot (\alpha \cdot y )
&= (a \cdot \alpha)\cdot (b\otimes x)-a \cdot (\alpha \cdot b\otimes x) \\
&= (a \cdot \alpha) b\otimes x- (a \cdot \alpha) \otimes b \cdot x-a  (\alpha \cdot b)\otimes x+ a\otimes (\alpha \cdot b)\cdot x \\
&=[ (a \cdot \alpha) b\otimes x- a  (\alpha \cdot b)\otimes x]  \\
&+[  (a \cdot \alpha) \otimes b \cdot x- a\otimes (\alpha \cdot b)\cdot x ],
\end{align*}
and each bracket clearly belongs to $ I_0 $. Other compatibility conditions are straightforward.

In particular, $ Y= \underbrace{\mathcal{A}\hat{\otimes}_{\mathfrak{A}} \mathcal{A}\hat{\otimes}_{\mathfrak{A}} \cdots\hat{\otimes}_{\mathfrak{A}}\mathcal{A}}_{k-times}\hat{\otimes}_{\mathfrak{A}}X $ is a commutative Banach \\
$ \mathcal{A} $-$  \mathfrak{A}$-bimodule,
and so is $Y^*$. The  isometric isomorphism in (i) now follows from  \cite[Exercise 5.3.1]{R}.  Finally, replacing cochains with module cochains, an argument almost idential to that of the proof of Theorem \cite[Theorem 2.4.6]{R} shows (ii). 
\end{proof}

\begin{cor}\label{cor1} Under the above assumptions, the seminormed spaces
 $ \mathcal{H}^{k}_{ \mathfrak{A}} ( \mathcal{A},X^*)$ and $\mathcal{H}^k ( \frac{\mathcal{A}}{J}, X^*) $ are isometrically isomorphic.
\end{cor}
\begin{proof} We have,
\begin{align*}
\mathcal{H}^{k}_{ \mathfrak{A}} ( \mathcal{A},X^*)
&\simeq  \mathcal{H}^1_{\mathfrak{A}} ( \mathcal{A},\mathcal{L}^{k-1}_{\mathfrak{A}}(\mathcal{A},X^*)) \quad \quad\quad\quad\quad\quad\quad\:\:\text{by Proposition \ref {red}} \\
&\simeq  \mathcal{H}^1 ( \frac{\mathcal{A}}{J},\mathcal{L}^{k-1}_{\mathfrak{A}}(\mathcal{A},X^*)) \:\quad \quad \quad \: \quad\quad\quad \quad\text{by Theorem \ref {main1}} \\
& \simeq  \mathcal{H}^1 ( \frac{\mathcal{A}}{J},(\underbrace{\mathcal{A}\hat{\otimes}_{\mathfrak{A}} \mathcal{A}\hat{\otimes}_{\mathfrak{A}}\cdots\hat{\otimes}_{\mathfrak{A}}\mathcal{A}}_{(k-1)-times}\hat{\otimes}_{\mathfrak{A}}X )^*) \\
&\simeq  \mathcal{H}^1 ( \frac{\mathcal{A}}{J},(\underbrace{\frac{\mathcal{A}}{J} \hat{\otimes} \frac{\mathcal{A}}{J} \hat{\otimes}\cdots\hat{\otimes}  \frac{\mathcal{A}}{J}}_{(k-1)-times}\hat{\otimes} X)^*) \qquad \quad\text{by Remark \ref {re}} \\
&\simeq  \mathcal{H}^1 ( \frac{\mathcal{A}}{J},\mathcal{L}^{k-1}(\frac{\mathcal{A}}{J},X^*))  \\
&\simeq \mathcal{H}^k ( \frac{\mathcal{A}}{J},X^*).
\end{align*}\end{proof}
\section{Applications to semigroup algebras}
\label{sec:2}
A discrete semigroup $ S $ is called an inverse semigroup if for every $ s\in S $, there exists a unique element $s^* \in S $ such that $ ss^*s=s $ and $ s^*ss^*=s^* $. An element
$ e\in S $ is called an idempotent if $ e=e^2=e^* $. The set of idempotents of $ S $ is denoted by
$ E $, which is a commutative subsemigroup of $ S $. In particular, $ \ell^1(E) $ could be regarded as a subalgebra of $\ell^1(S)  $, and thereby $\ell^1(S)  $ is a Banach $ \ell^1(E) $-bimodule with a compatible canonical action (by multiplication). 

There is a natural order on $ E $ defined by
$$ e_1 \leq e_2 \Longleftrightarrow e_1e_2=e_2e_1=e_1, \quad (e_1,e_2 \in E). $$
An inverse semigroup $ S $ is called a Clifford semigroup, if for every $ s\in S $,
$ ss^*=s^*s $, or equivalently, an inverse semigroup with each idempotent central. In this case $ E $ becomes a semilattice, that is, a commutative semigroup with each element an idempotent.

Let $ S $ be an inverse semigroup with the set of idempotents $ E $. 
 Let
$\ell^1(E)  $ act on $\ell^1(S)  $ by multiplication from right and  trivially  from left (induced by the augmentation character $ f $ on $ \ell^1(E)$), that is,
\begin{eqnarray}\label{act}
\delta_e\cdot \delta_s=\delta_s, \quad \delta_s\cdot \delta _e=\delta_{se}=\delta_s \ast \delta_e \quad (e \in E, s\in S),
\end{eqnarray}
where $ \delta _s $ is the point mass at $ s $. Here the closed ideal $ J $ (see section 1) is the closed linear span  of the set
$$\lbrace   \delta_{set}- \delta_{st}: s,t \in S, e\in E \rbrace.  $$

We consider an equivalence relation on $ S $ defined by
\begin{eqnarray}\label{equ}
s \sim t \Longleftrightarrow \delta_s - \delta_t \in J \quad (s,t \in S).
\end{eqnarray}
Since $ E $ is a semilattice, the discussion before \cite[Theorem 2.4]{ABE} shows that $ S /\sim $ is a discrete group. In this case, by the proof of \cite[Theorem 3.3]{RASE}, we observe that
$ \frac{\ell^1(S)}{J} \simeq \ell^1(S /\sim )$ as (commutative) $\ell^1(E)  $-bimodules. The discrete group $ S /\sim$ is the same as the maximal group homomorphic image $ G_S $ of $ S $.

E. Nasrabadi and the second author in \cite {NP} showed that, for a commutative inverse semigroup $ S $ with the set of idempotents $ E $, $ \mathcal{H}^1_{ \ell^1(E)} ( \ell^1(S),\ell^1(S)^{(2n+1)}) $ is trivial, for each $ n \in \mathbb{N} $. In general, when $ S $ is not commutative and $\ell^1(E)  $ acts on $\ell^1(S)  $ by multiplication from right and trivially from left, $\ell^1(S)  $ is not necessarily commutative as a Banach $ \ell^1(E) $-bimodule. However, if one considers $ \frac{\ell^1(S)}{J} $  as a Banach $ \ell^1(E) $-bimodule, under the above actions, $ \frac{\ell^1(S)}{J} $ is a commutative $\ell^1(S)$-$ \ell^1(E) $-bimodule. We show that $ \mathcal{H}^1_{ \ell^1(E)} ( \ell^1(S),(\frac{\ell^1(S)}{J})^{(2n+1)})  $
is trivial, for each $ n \in \mathbb{N} $.

Let us recall the following known result \cite[Theorems 2.2, 3.3]{P}.
\begin{pro} \label{dr}
Let $ G $ be a locally compact group. Then $ \mathcal{H}^2 ( L^1(G),L^1(G)^{(2n+1)})  $ is a Banach space, for each $ n \in \mathbb{N} \cup \lbrace 0 \rbrace$. Moreover, when 
$ G $ is discrete, \\
 $ \mathcal{H}^1 ( \ell^1(G),\ell^1(S)^{(2n+1)}) =0 $, for every $ G $-set $ S $ and for each $ n \in \mathbb{N}\cup \lbrace 0 \rbrace $.
\end{pro}

\begin{thm}
Let $ S $ be an inverse semigroup with the set of idempotents $ E $. Let $\ell^1(E)  $ act on $\ell^1(S)  $ by multiplication from right and trivially from left. Then,
$$\mathcal{H}^1_{\ell^1(E)} (\ell^1(S), \ell^1(G_S)^{(2n+1)})=0,$$ for each $ n \in \mathbb{N}\cup \lbrace 0 \rbrace $.
\end{thm}
\begin{proof}
 For every $s \in S  $, letting $ e =s^*s$, we have  $ \delta_s -\delta_{se} \in J $, thus $ \ell^1(G_S) =\frac{\ell^1(S)}{J} $ is a commutative Banach  $\ell^1(S)$-$\ell^1(E) $-bimodule and by Proposition \ref {dr} we have
\begin{align*}
\mathcal{H}^1_{\ell^1(E)} (\ell^1(S),\ell^1(G_S)^{(2n+1)})
&\simeq \mathcal{H}^1 (\frac{\ell^1(S)}{J},(\frac{\ell^1(S)}{J})^{(2n+1)}) \quad \qquad \text{(by Theorem \ref {main1})} \\
&\simeq\mathcal{H}^1 (\ell^1(S /\sim),(\ell^1(S /\sim))^{(2n+1)})    \\
&\simeq\mathcal{H}^1 (\ell^1(G_S),(\ell^1(G_S))^{(2n+1)})=0.
\end{align*}\end{proof}

This provides an affirmative answer to  the module derivation problem, asking if 
every module derivation $ D: \ell^1(S) \rightarrow \frac{\ell^1(S)}{J}$ is inner (compare with  \cite[Corollary 3.4]{EA}).

Let $  \mathcal{C}$ be a bicyclic inverse semigroup generated by $  p$
and $  q$, that is,
$$ \mathcal{C}=\lbrace p^mq^n : m, n\geq 0    \rbrace, \quad (p^mq^n)^*=p^nq^m.$$
The multiplication in $  \mathcal{C}$ is defined by
$$(p^mq^n)(p^{m^\prime} q^{n^\prime})=p^{m-n+\max \lbrace n,m^\prime  \rbrace}  q^{m^\prime -n^\prime +\max \lbrace n,m^\prime  \rbrace}.$$
The set of idempotents of $  \mathcal{C}$ is $ E_{  \mathcal{C}}=\lbrace p^nq^n : n= 0, 1, ...      \rbrace $, which is totally ordered with the following order,
$$p^nq^n \leq p^mq^m \Longleftrightarrow m \leq n.  $$

Suppose  $ \ell^1(E_{\mathcal{C}}) $ acts on $ \ell^1(\mathcal{C}) $ trivially from left and by multiplication from right. Consider the equivalence relation  on $ \ell^1(\mathcal{C}) $  as in (\ref {equ}). In this case, it is shown in \cite {ABE} that $ \mathcal{C}/\sim  $ is isomorphic to $ \mathbb{Z} $ and hence we have
$$   \dfrac{\ell^1(\mathcal{C}) }{J} \simeq \ell^1(\mathcal{C}/\sim) \simeq  \ell^1( \mathbb{Z}).$$
\begin{pro} Let $ X $ be a  commutative Banach $ \ell^1(\mathcal{C})$-$\ell^1(E_{\mathcal{C}})$-bimodule satisfying the assumptions of Corollary \ref {cor1}.
 Then, for each $ k \in \mathbb{N} $, $$ \mathcal{H}^k_{\ell^1(E_{\mathcal{C}})} (\ell^1(\mathcal{C}) ),X^*)=0. $$
\end{pro}
\begin{proof}
Since $ \dfrac{\ell^1(\mathcal{C} )}{J} =\ell^1(\mathbb{Z} )  $ is an $ \mathfrak{L}_{1 }$-space, by Corollary \ref{cor1}, 
$$
\mathcal{H}^k_{\ell^1(E_{\mathcal{C}})} (\ell^1(\mathcal{C} ), X^*)
\simeq{H}^k  (\dfrac{\ell^1(\mathcal{C} )}{J}, X^* )
\simeq{H}^k (\ell^1(\mathbb{Z} ),X^*)=0 , 
$$
where the last equality follows from amenability of $ \mathbb{Z}  $.
\end{proof}

Next we generalize \cite[Corollary 3.5]{NP} from Clifford semigroups to arbitrary inverse semigroups. 

Let  $ \Omega $ be a locally compact Hausdorff space. Then the character space of    $ C_0(\Omega)^{**} $   is called  the hyper-Stonean envelope of $ \Omega $ and  is denoted by $ \tilde{\Omega}$. By the Gelfand transform, $ C(\tilde{\Omega})$ is isometrically isomorphic to $ (C_0(\Omega))^{**} $ and therefore $  M(\Omega)^*\simeq  (C_0(\Omega))^{**} \simeq C(\tilde{\Omega}) $, thus
$$ M(\Omega)^{**} \simeq C(\tilde{\Omega})^* \simeq M(\tilde{\Omega}).$$

\begin{pro}\cite [Proposition 5.7]{DLS}\label{1.4}
Let  $ G $ be a locally compact group and let $ X $ be a Banach $  C_0(G)$-submodule of $  M(G) $. Suppose that  the character space of the commutative
C*-algebra $ X^* $ is $ \Phi_{X}  $. Then,
\begin{enumerate}
\item[(i)] $\Phi_{X}  $ is a clopen subset of $ \tilde{G} $,
\item[(ii)]If  $ X $ is a subalgebra (ideal) of the Banach algebra $ (M(G),\ast )$, then \\
$ (X^{**},\square)\simeq (M(\Phi_{X}),\square)  $ is a closed subalgebra (ideal) of $ (M(\tilde{G})  ,\square) $.
\end{enumerate}
\end{pro}
\begin{cor} \label{1.5}
Let  $ G $ be a locally compact group. Then $( L^1(G)^{**},\square )\simeq (M(\Phi),\square)$ is an ideal of $ (M(\tilde{G})  ,\square) $, where $\Phi $ is the character space of $ L^\infty(G) $.
\end{cor}
\begin{proof}
If $ f \in L^1(G) $ and $ \varphi \in C_0(G) $ and $ fm $ is the image   of $  f$ in $  M(G) $, where $ m $  is the  Haar measure, then the actions of $  C_0(G)$ on  $   L^1(G) $  is defined by
$$f \cdot \varphi :=(fm) \centerdot \varphi, \qquad   \qquad \varphi  \cdot f := \varphi \centerdot (fm),$$
where  $ \centerdot $ is the original action of $  C_0(G)$ on $  M(G) $. Thus $ L^1(G) $ is a Banach $  C_0(G)$-submodule of $  M(G) $, and one may apply Proposition \ref{1.4}.
\end{proof}
\begin{cor}\label{1.6}
Let  $ G $ be a locally compact group. Then for every $  n \in \mathbb{N} $, we have $ L^1(G)^{(2n)}\simeq M(K_n) $, for some compact Hausdorff space $ K_n $. Also $ M(K_n) $ is a Banach $ L^1(G)  $-bimodule.
\end{cor}
\begin{proof}
By  Corollary \ref{1.5}, $( L^1(G)^{**},\square )\simeq (M(\Phi),\square)$, where $\Phi $ is the character space of $ L^\infty(G) $. Also $ (M(\Phi),\square)^{**}\simeq (M(\tilde{\Phi}),\square) $, where $ \tilde{\Phi} $ is the hyper-Stonean envelope of $ \Phi $. Thus 
$$ (L^1(G)^{(4)},\square )\simeq ( L^1(G)^{**},\square )^{**}\simeq (M(\Phi),\square)^{**}\simeq (M(\tilde{\Phi}),\square).$$ 
Inductively, one could show that, for each $ n\in \mathbb{N} $, $ L^1(G)^{(2n)}\simeq M(K_n), $ for some compact  Hausdorff space $ K_n $. Now the canonical Banach $ L^1(G)  $-bimodule structure of $ L^1(G)^{(2n)} $  gives
$ M(K_n) $ a Banach module structure.
\end{proof}
\begin{cor}\label{2n}
Let  $ G $ be a locally compact group. Then for every $  n \in \mathbb{N} $, \\
$ L^1(G)^{(2n)}\simeq M(K_n)$ is an $ \mathfrak{L}_{1 }$-space.
\end{cor}
\begin{proof}
It is immediate from \cite[Exercise 2.11]{Ry}.
\end{proof}
\begin{thm}
Let $ S $ be an inverse semigroup with the set of idempotents $ E $. Let $\ell^1(E)  $ act on $\ell^1(S)  $ by multiplication from right and trivially from left. Then, for each $ n \in \mathbb{N} \cup \lbrace 0 \rbrace$,
$\mathcal{H}^2_{\ell^1(E)} (\ell^1(S),\ell^1(G_S)^{(2n+1)})$ is a Banach space.
\end{thm}
\begin{proof}
Fix  $ n \in \mathbb{N} \cup \lbrace 0 \rbrace$, and put $ X=\ell^1(G_S)^{(2n)}  $. Since the left action of  $\ell^1(E)  $  on  $ \ell^1(G_S) $  is trivial, so is the left action of  $\ell^1(E)  $  on $ X $. The corresponding ideal $ J $ is spanned by the elements of the form $  \delta_{set}- \delta_{st}$,  with  $ s,t \in S, e\in E $. Since $$  \delta_{set}- \delta_{st}=  \delta_{sett^*t}- \delta_{st} = \delta_{stt^*et}- \delta_{st},$$
and $ t^*et $ is an idempotent, say $ e^\prime $, we have $  \delta_{set}- \delta_{st}=\delta_{re^\prime }- \delta_r, $ for $ r=st. $ Thus the elements of the form $ \delta_r\cdot \delta_{e^\prime }- f(e^\prime)\delta_r $ span $ J $, where $ f $ is the augmentation character on $\ell^1(E)  $.

Next, for $ n=0 $, the Banach space $\frac{\ell^1(S)}{J} \simeq\ell^1(G_S)  $ is a unital Banach algebra and an $ \mathfrak{L}_{1 }$-space. Also $ \frac{\ell^1(S)}{J}\hat{\otimes} \ell^1(G_S) \simeq \ell^1(G_S \times G_S)   $ is an $ \mathfrak{L}_{1 }$-space.
 Therefore, by   Corollary \ref {cor1},
\begin{align*}
\mathcal{H}^2_{\ell^1(E)} (\ell^1(S),\ell^\infty(G_S))
&\simeq \mathcal{H}^2 (\frac{\ell^1(S)}{J},\ell^\infty(G_S))  \\
&\simeq\mathcal{H}^2 (\ell^1(G_S),\ell^\infty(G_S)).
\end{align*}
 By   Proposition \ref {dr}, the last space is a Banach space.

 For $ n \geq 1 $, by Corollary \ref {2n}, $X  $ is an $ \mathfrak{L}_{1 }$-space. Hence, by Corollary \ref {cor1},
$$
\mathcal{H}^2_{\ell^1(E)} (\ell^1(S),X^*)
\simeq \mathcal{H}^2 (\frac{\ell^1(S)}{J},X^*)  \\
\simeq\mathcal{H}^2 (\ell^1(G_S),(\ell^1(G_S))^{(2n+1)}).
$$
Again, by   Proposition \ref {dr}, the last space is a Banach space. 
  \end{proof}

\end{document}